\documentclass[12pt,a4paper]{amsart}
\usepackage{mathrsfs}
\usepackage{amssymb,amsmath,amsthm,color}
\usepackage{graphicx,mcite, cite}
\usepackage{hyperref}
\usepackage{url}
\usepackage{setspace}
\usepackage{enumerate}
\newtheorem{theorem}{Theorem}[section]
\newtheorem{lemma}[theorem]{Lemma}
\newtheorem{proof of lemma}[theorem]{Proof of Lemma}

\newtheorem{corollary}[theorem]{Corollary}

\theoremstyle{definition}
\newtheorem{definition}[theorem]{Definition}

\newtheorem{remark}[theorem]{Remark}

\numberwithin{equation}{section}

\begin{document}

\title[Twisted spherical means]
{Spherical means on M\'{e}tivier groups and support theorem}

\author{Rupak Kumar Dalai, Somnath Ghosh and R.K. Srivastava}

\address{Department of Mathematics, Indian Institute of Technology, Guwahati, India 781039.}
\email{rupak.dalai@iitg.ac.in, gsomnath@iitg.ac.in, rksri@iitg.ac.in}


\date{\today}


\begin{abstract}
Let $Z_{r, R}$ be the space of continuous functions on the annulus $B_{r, R}$
in $\mathbb C^n$ whose $\lambda$-twisted spherical mean, in the set up of the
M\'{e}tivier group, vanishes over the spheres $S_s(z)\subset B_{r, R} $ with
ball $B_r(0)\subseteq B_s(z).$ We characterize the spherical harmonic coefficients
of functions in $Z_{r, R},$ eventually, in terms of polynomial growth, by which
we infer support theorem. Further, we prove that non-harmonic complex cone and
the boundary of a bounded domain are sets of injectivity for the $\lambda$-twisted
spherical means.
\end{abstract}

\maketitle

\section{Introduction}\label{section1}
In a remarkable result, Helgason proved a support theorem for continuous
functions having polynomial growth whose spherical mean vanishes over the spheres
surrounding a ball. In other words, let $\mu_s$ be the normalized surface measure
on the sphere $S_s^{n-1}.$ If $f$ is a continuous function on $\mathbb R^n,(n\geq2)$
such that $|x|^kf(x)$ is bounded for each non-negative integer $k,$ then $f$ is
supported in the ball $B_r(0)$ if and only if $f\ast\mu_s(x)=0, \forall~x\in\mathbb R^n$
and $\forall s>|x|+r,$ (see \cite{H}).

\smallskip

Later in \cite{EK}, Epstein and Kleiner generalized the Helgason's support
theorem significantly, by characterizing the space of all continuous functions on $\mathbb R^n,$
whose spherical mean vanishes over all spheres surrounding a ball, in terms of spherical harmonic
coefficients having polynomial growth. This result was first proved by Globevnik \cite{G}
in the plane.

\smallskip

Let $H_k$ be the restriction of the space of homogeneous harmonic polynomials
of degree $k$ to the unit sphere $S^{n-1},$ and $\{Y_k^l:l=1,\ldots,d_k\}$
is an orthonormal basis for $H_k.$ Then $f\in C(\mathbb R^n)$
can be expressed as
\[f(x)=\sum_{k=0}^{\infty}\sum_{l=1}^{d_k}~a_{kl}(\rho)~Y_k^l(\omega),\]
where $x=\rho\omega$ and $\rho =|x|.$ In \cite{EK}, authors had shown that
 $f\ast\mu_s(x)=0$ for all $x\in\mathbb R^n$ and $s>|x|+B$
as long as $a_{kl}\in\text{span}\{\rho^{k-n-2i}:i=0,1,\ldots,k-1\},$
whenever $\rho>B.$

\smallskip

Consequently, the support theorem is an immediate corollary of the above
result \cite{EK}. For other related work, we refer to \cite{BQ,CQ,Q1,V1,V2}.

\smallskip

Further, in the article \cite{NT2}, Thangavelu and Narayanan proved an analogue of the Helgason's
support theorem for the twisted spherical mean (TSM) for certain Schwartz class functions
on $\mathbb C^n.$ In \cite{RS}, authors characterized the space of all continuous functions
on $\mathbb C^n$ having TSM mean vanishes over the spheres surrounding a ball,
and proved an exact analogue of the Helgason's support  theorem for the twisted
spherical mean on $\mathbb C^n ~(n\geq2).$ For $n=1,$ authors have proved a
stronger result relaxing decay condition.

\smallskip
In Section \ref{section3}, we consider M\'{e}tivier group for
proving necessary conditions for a function to be in $Z^*_{r, R},$ the subspace
of ceratin smooth functions on $B_{r, R}.$ This result imitate a support theorem
for type functions. Further, we derive that the non-harmonic complex
cones in $\mathbb{C}^n$ are sets of injectivity for the $\lambda$-twisted spherical mean
for the class of continuous functions on M\'{e}tivier group. For a brief history of work
related to sets of injectivity for TSM on the Heisenberg group, we refer
to\cite{AR,NT1,Sri1,Sri2,Sri3,Sri4}.

\smallskip

Finally, in Section \ref{section4}, we prove analogous results on $H$-type group,
which is a special case of M\'{e}tivier group. We prove sufficient
condition for a function to be in $Z_{r, R},$ and derive the support theorem
and Heche-Bochner identity for the $\lambda$-twisted spherical mean. Then we prove
that the boundary of a bounded domain is a set of injectivity for $\lambda$-twisted
spherical mean.

\smallskip

We would like to mention that the results in the case of M\'{e}tivier group
is of restrictive nature with those in Heisenberg group due the fact that
the symplectic bilinear form appears in the group action of the M\'{e}tivier
group cannot be made $U(n)$-invariant, in general, due to higher dimensional
center of the M\'{e}tivier groups. This fact will be reflected as the distinct
eigenvalues of the corresponding symplectic matrix $U_\lambda$ that arises in Section
\ref{section3}. However, we prove that the symplectic bilinear form for $H$-type
group is similar to that of the Heisenberg group upto an orthogonal transformation.

\section{Preliminaries}\label{section2}
Let $G$ be connected, simply connected Lie group with real step two nilpotent
Lie algebra $\mathfrak g.$ Then $\mathfrak g$ has the orthogonal decomposition
$\mathfrak g=\mathfrak b\oplus\mathfrak z,$ where $\mathfrak z$ is the center
of $\mathfrak g.$ Since $\mathfrak g$ is a nilpotent, the exponential map
$\exp:\mathfrak g\rightarrow G$ is surjective, and hence $G$ can be parameterized
by $\mathfrak g,$ endowed with the exponential coordinates. Now, we can identify
$X+T\in \mathfrak{b}\oplus \mathfrak{z}$ with $\exp(X+T)$ and denote it by
$(X,T)\in \mathbb{R}^d \times \mathbb{R}^m$. Since $[\mathfrak{b},\mathfrak{b}]\subseteq\mathfrak{z}$
and $[\mathfrak{b},[\mathfrak{b},\mathfrak{b}]]=0$, by the Baker-Campbell-Hausdorff
formula, the group law on $G$ can be expressed as \[(X,T)(Y,S)=(X+Y,T+S+\frac{1}{2}[X,Y]),\]
where $X,Y\in \mathfrak{b}$ and $T,S\in\mathfrak{z}$.
Now, for $\omega\in\mathfrak z^\ast,$ consider the skew-symmetric bilinear form
$B_\omega$ on $\mathfrak b$ by $B_\omega(X,Y)=\omega\left([X,Y]\right).$
Let $m_\omega$ be the orthogonal complement of
$r_\omega=\left\{X\in\mathfrak b:B_\omega(X,Y)=0,~\forall~ Y\in\mathfrak b\right\}$ in
$\mathfrak b.$ Then $B_\omega$ is called a non-degenerate bilinear form when
$r_\omega$ is trivial.

\smallskip

In this article, we discuss some special type of step two nilpotent Lie groups.

\smallskip

\textbf{M\'{e}tivier groups:} We say the group $G$ is M\'{e}tivier group if $B_\omega$ is
non-degenerate for all non-zero $\omega\in\mathfrak z^\ast.$ In this case, $d=2n,$ even.
Let $B_1,\ldots, B_{2n}$ and $Z_1,\ldots, Z_m$ be orthonormal bases for $\mathfrak{b}$
and $\mathfrak{z}$ respectively. Since $[\mathfrak{b},\mathfrak{b}]\subseteq \mathfrak{z},$
there exist scalars $U_{j,l}^{(k)}$ such that
\[[B_j,B_l]=\sum_{k=1}^m U_{j,l}^{(k)}Z_k,\quad 1\leq j,l\leq 2n.\]
For $1\leq k\leq m,$ define ${2n\times 2n}$ skew-symmetric matrices by
$U^{(k)}=(U_{j,l}^{(k)}).$ Then the group law  for the M\'{e}tivier group can be expressed as
\begin{align}\label{exp4}
(x,t).(\xi,\tau)=\binom{x_i+\xi_i,~i=1,\ldots,2n}{~t_j+\tau_j+\frac{1}{2}
\langle x, U^{(j)}\xi\rangle,~j=1,\ldots,m },
\end{align}
where $x,\xi\in\mathbb R^{2n}$ and $t,\tau \in \mathbb R^m.$ For
$x=(x_1,\ldots,x_n,y_1,\ldots,y_n)\in \mathbb R^{2n},$ write
$z=(x_1+iy_1,\ldots,x_n+iy_n)=(z_1,\ldots,z_n)$ and say, $z$ be the complexification of $x.$
Let $z,w\in \mathbb C^n$ be the complexification of $x,\xi\in \mathbb R^{2n}.$ If we fix the
notation $U^{(j)}w$ for the complexification of $U^{(j)}\xi,$ then (\ref{exp4}) can be
simplified to
\begin{align}\label{exp8}
(z,t).(w,\tau)=\binom{z+w}{~t_j+\tau_j+\frac{1}{2}\text{ Re}\,(z\cdot\overline{U^{(j)}w}),~j=1,\ldots,m }.
\end{align}

\smallskip

\textbf{$H$-type groups:}
Suppose $\mathfrak{g}$ is endowed with an inner product $\langle\cdot,\cdot\rangle$ such that
for each $Z\in\mathfrak{z},$ the map $J_Z: \mathfrak{b}\rightarrow\mathfrak{b}$ defined by
$\langle J_Z(X),Y \rangle=\langle Z,[X,Y]\rangle$ for $X,Y\in\mathfrak{b},$ satisfies
$J^T_Z=-J_Z.$ We say that $\mathfrak{g}$ is $H$-type if $ J^2_Z=-|Z|^2I,$
whenever $Z\in\mathfrak{z}.$
Hence it follows that $J_ZJ_{Z'}+J_{Z'}J_Z=-2\langle Z,Z' \rangle I ~\mbox{for~all}~Z, Z'\in\mathfrak{z},$
where $I$ denotes the identity mapping. A connected, simply connected Lie group
$G$ with $H$-type Lie algebra is called Heisenberg type (or $H$-type) group. The
$H$-type groups, introduced by A. Kaplan \cite{K}, are examples of M\'{e}tivier group.
However, there are M\'{e}tivier groups that differ from the $H$-type groups.
For more details, see \cite{M,MS}.

\begin{theorem}\emph{\cite{BU}}\label{th1}
Let $G$ be connected, simply connected Lie group with real step two nilpotent Lie algebra
$\mathfrak g.$ Then $G$ is a $H$-type group if and only if $G$ is isomorphic to $\mathbb R^{2n+m}$ with
the group law (\ref{exp4}) and the matrices $U^{(1)},\ldots,U^{(m)}$ satisfies the following conditions:\\
$(a)$ $U^{(j)}$ is skew-symmetric $2n\times 2n$ orthogonal matrix, for $~j=1,\ldots,m.$ \\
$(b)$ $U^{(j)}U^{(l)}+U^{(l)}U^{(j)}=0$ ~for all ~$j,l=1,\ldots, m$ with $j\neq l$.
\end{theorem}
For $\lambda\in\mathfrak z\smallsetminus\{0\},$ it follows from Theorem \ref{th1}
that $\sum_{j=1}^m\lambda_jU^{(j)}=|\lambda|V,$ where $V$ is an orthogonal matrix.
This fact will enable us to deduce that $\lambda$-twisted spherical mean on $H$-type
group is similar to the $|\lambda|$-twisted spherical mean on the Heisenberg group.

\smallskip

Let $\mu_s$ be the normalized surface measure on the set $ \{(z, 0):~|z|=s\} \subset G.$
Then the partial spherical means of a function $F\in L^1_{loc}(G)$ can be defined by
\begin{align}\label{exp9}
F \ast \mu_s (z, t) = \int_{|w|=s}F((z,t)(-w,0))~d\mu_s(w).
\end{align}
Let $$F^\lambda(z)= \int_{\mathbb R^m} F(z,t)e^{i \lambda \cdot t} dt,$$ be the
inverse Fourier transform of $F$ in the $t$ variable. Then
\begin{equation}\label{exp10}
(F \ast \mu_s)^\lambda(z)=\int_{|w| = s}~F^\lambda (z-w)
e^{\frac{i}{2}\sum_{j=1}^m\lambda_j\text{Re}\,(z\cdot\,\overline{U^{(j)}w})}~d\mu_s(w).
\end{equation}
Define the $\lambda$-twisted spherical means of $f\in L^1(\mathbb C^n)$  by
\begin{align}\label{exp14}
f \times_\lambda \mu_s(z)=\int_{|w|=s}~f(z-w)~e^{\frac{i}{2}\sum_{j=1}^m\lambda_j
\text{Re}\,(z\cdot\,\overline{U^{(j)}w})}~d\mu_s(w).
\end{align}
From (\ref{exp10}) we get $(F \ast \mu_s)^\lambda=F^\lambda\times_\lambda\mu_s.$
Thus, the partial spherical mean $F\ast\mu_s$ on the M\'{e}tivier group $G$ can be
thought of the $\lambda$-twisted spherical mean $F^\lambda\times_\lambda\mu_s.$
Note that $\lambda$-twisted spherical mean (\ref{exp14}) is the complexification of the mean
\begin{align}\label{exp16}
f\times_\lambda \mu_s(x)=\int_{|\xi|=s}~f(x-\xi)~e^{\frac{i}{2}
\sum_{j=1}^m\lambda_j\langle x,U^{(j)}\xi\rangle}~d\mu_s(\xi).
\end{align}

\begin{definition}
Let $B_{r, R}=\{z\in \mathbb C^n: r<|z|< R\}$ be an open annulus in $\mathbb C^n,$
where $0\leq r<R\leq\infty.$  Let $Z_{r, R}$ be the space of all continuous
functions $f$ on $B_{r, R}$ such that $f\times_\lambda\mu_s(z)=0$ on the spheres
$S_s(z)\subset B_{r, R}$ and the ball $B_r(0)\subseteq B_s(z).$
\end{definition}

Let $Z_{r, R}^\infty$ be space of all smooth functions in $Z_{r, R}.$
Consider a smooth non-negative radial function $\phi$ on $\mathbb C^n,$
supported in $B_1(0)$ and $\int_{\mathbb C^n}\phi=1.$ When $\epsilon>0,$
write $\phi_\epsilon(z)=\epsilon^{-2n}\phi(\frac{z}{\epsilon}).$
For $f\in Z_{r, R}^\infty,$ define $S_\epsilon(f)$ by
\[S_\epsilon(f)(z)=\int_{\mathbb C^n}f(z-w)\phi_\epsilon(w)
e^{\frac{i}{2}\sum_{j=1}^m\lambda_j\text{Re}\,(z\cdot\,\overline{U^{(j)}w})}dw.\]
Then we can deduce that $S_\epsilon(f)\in Z_{r+\epsilon, R-\epsilon}^\infty.$
Since $\text{ supp }\phi_\epsilon\subseteq B_\epsilon(0),$ and
\begin{align*}
S_\epsilon(f)(z)-f(z)=\int_{|w|\leq\epsilon}&\phi_\epsilon(w)
e^{\frac{i}{2}\sum_{j=1}^m\lambda_j\text{Re}\,(z\cdot\,\overline{U^{(j)}w})}
(f(z-w)-f(z))dw\\
&+\int_{|w|\leq\epsilon}
(e^{\frac{i}{2}\sum_{j=1}^m\lambda_j\text{Re}\,(z\cdot\,\overline{U^{(j)}w})}-1)
f(z)\phi_\epsilon(w)dw,
\end{align*}
together with $f$ is continuous, letting $\epsilon$ goes to $0,$ it follows that
$S_\epsilon(f)$ converges to $f$ locally uniformly. Thus,
without loss of generality, we can assume the functions in $Z_{r, R}$
are smooth.

\smallskip

Since $Z_{r, R}$ is closed under small translation, it follow that
$Z_{r, R}$ will be invariant under the action of appropriate vector fields on $G.$

\smallskip

The left-invariant vector fields on $G$ are
\[X_j=\dfrac{\partial}{\partial x_j}+ \frac{1}{2}\sum_{k=1}^m\left(\sum_{l=1}^n
\left(x_lU_{l,j}^{(k)}+y_lU_{n+l,j}^{(k)}\right)\right)\dfrac{\partial}{\partial t_k},\]
\[Y_j=\dfrac{\partial}{\partial y_j}+ \frac{1}{2}\sum_{k=1}^m\left(\sum_{l=1}^n
\left(x_lU_{l,n+j}^{(k)}+y_lU_{n+l,n+j}^{(k)}\right)\right)\dfrac{\partial}{\partial t_k},\]
\[T_k=\dfrac{\partial}{\partial t_k},\text{ where }k=1,\ldots,m, ~j=1,\ldots,n,\]
and $(x,t)=(x_1,\ldots, x_n,y_1,\ldots,y_n,t_1,\ldots,t_m)\in \mathbb R^{2n}\times \mathbb R^m.$
In fact, they generate a basis of the Lie algebra of the M\'{e}tivier group $G$.
Given $U^{(s)}$'s are skew-symmetry, we obtain the following commutation relations
\[[X_i,X_j]=\sum_{k=1}^mU_{i,j}^{(k)}\dfrac{\partial}{\partial t_k},~[Y_i,Y_j]
=\sum_{k=1}^mU_{n+i,n+j}^{(k)}\dfrac{\partial}{\partial t_k},\text{ for } ~i,j=1,\ldots,n.\]
Since $U^{(1)},\ldots,U^{(m)}$ are linearly independent, the dimension of the space
spanned by $\{(U_{i,j}^{(1)},\ldots,U_{i,j}^{(m)}):~i,j=1,\ldots,n\}$ will be $m.$

\smallskip

Now, for $1\leq j\leq n,$ define
\begin{align*}
Z_j&=\frac{1}{2}(X_j-iY_j) \\
&=\dfrac{\partial}{\partial z_j}+\frac{1}{4}
\sum_{k=1}^m\sum_{l=1}^n\left\{x_l\left(U_{l,j}^{(k)}-iU_{l,n+j}^{(k)}\right)+
y_l\left(U_{n+l,j}^{(k)}-iU_{n+l,n+j}^{(k)}\right)\right\}\dfrac{\partial}{\partial t_k},\\
\bar{Z}_j&=\frac{1}{2}(X_j+iY_j) \\
&=\dfrac{\partial}{\partial \bar{z}_j}+\frac{1}{4}
\sum_{k=1}^m\sum_{l=1}^n\left\{x_l\left(U_{l,j}^{(k)}+iU_{l,n+j}^{(k)}\right)+
y_l\left(U_{n+l,j}^{(k)}+iU_{n+l,n+j}^{(k)}\right)\right\}\dfrac{\partial}{\partial t_k}.
\end{align*}
Consider the function $F$ on
$G=\mathbb C^n\times\mathbb R^m$ of type $F(z,t)=e^{i\lambda.t}f(z),$ where
$\lambda\in \mathfrak z\smallsetminus\{0\}.$ Then the vector fields $Z_j$ and $\bar{Z_j}$
reduce to
\begin{align}
Z_j^{\lambda}&=\dfrac{\partial}{\partial z_j}+\frac{1}{4}\sum_{l=1}^n
\left\{\left(\beta_l^{\lambda}+i\alpha_l^{\lambda}\right)z_l+
\left(-\beta_l^{\lambda}+i\alpha_l^{\lambda}\right)\bar{z}_l\right\} \nonumber \\
&=\dfrac{\partial}{\partial z_j}+\frac{1}{4}\nu_j\bar{z}_j
+\frac{1}{4}\sum_{\substack{l=1 \\ (l\neq j)}}^n (\eta_lz_l+\nu_l\bar{z}_l), \label{exp5} \\
\bar{Z}_j^{\lambda}&=\dfrac{\partial}{\partial \bar{z}_j}+\frac{1}{4}\sum_{l=1}^n
\left\{\left(\bar{\beta}_l^{\lambda}+i\bar{\alpha}_l^{\lambda}\right)z_l+
\left(-\bar{\beta}_l^{\lambda}+i\bar{\alpha}_l^{\lambda}\right)\bar{z_l}\right\} \nonumber \\
&=\dfrac{\partial}{\partial \bar{z}_j}-\frac{1}{4}\bar{\nu}_jz_j
-\frac{1}{4}\sum_{\substack{l=1 \\ (l\neq j)}}^n (\bar{\nu}_lz_l+\bar{\eta}_l\bar{z}_l), \label{exp6}
\end{align}
since $\eta_j=0,$ where we denote \[\alpha_l^{\lambda}=\frac{1}{2}\sum_{k=1}^m\lambda_k\left(U_{l,j}^{(k)}-iU_{l,n+j}^{(k)}\right),
\beta_l^{\lambda}=\frac{1}{2}\sum_{k=1}^m\lambda_k\left(U_{n+l,j}^{(k)}-iU_{n+l,n+j}^{(k)}\right)\]
and $\eta_l=\beta_l^{\lambda}+i\alpha_l^{\lambda},$ $\nu_l=-\beta_l^{\lambda}+i\alpha_l^{\lambda}$ for $1\leq l\leq n.$

\smallskip

The differential operators $Z_j^\lambda$ and $\bar{Z}_j^\lambda$ play a role of
left-invariant vector fields for $\lambda$-twisted convolution on $\mathbb C^n.$
That is, \[Z_j^\lambda(f\times_\lambda\mu_s)=Z_j^\lambda f\times_\lambda\mu_s
\text{ and } \bar{Z}_j^\lambda (f\times_\lambda\mu_s)=\bar{Z}_j^\lambda f\times_\lambda\mu_s.\]
As an effect, if $f\in Z_{r, R},$ then $Z_j^\lambda f$ and $\bar{Z}_j^\lambda f$ both are in $Z_{r, R}.$

\subsection{Bi-graded spherical harmonics}
We require the bi-graded spherical harmonic expansion of continuous function on
$\mathbb C^n.$  See \cite{D, Gr1, R, T2} for details.

For $p, q \in \mathbb Z_+,$ the set of all non-negative integers, let $P_{p,q}$ denote
the space of all polynomials $P$ in $z$ and $\bar{z}$ of the form
\[P(z)=\sum_{|\alpha|=p} \sum_{|\beta|=q}c_{\alpha\beta}~z^\alpha \bar{z}^\beta .\]
Write $H_{p,q}=\{P\in P_{p,q}:\Delta P=0\},$ where $\Delta$ stands for the Laplacian
on $\mathbb C^n.$ The elements of $H_{p,q}$ restricted to the unit sphere $S^{2n-1}$
are called bi-graded spherical harmonics. Now, we identify $H_{p,q}$ as the space of
bi-graded spherical harmonics on $S^{2n-1}.$ Let $\{Y_j^{p,q}:1\leq j\leq d_{p,q}\}$
be an orthonormal basis of $H_{p,q}.$ By Peter-Weyl theorem the set
$\{Y_j^{p,q}:1\leq j\leq d_{p,q},~p,q\in\mathbb Z_+\}$ forms an orthonormal basis
for $L^2(S^{2n-1}),$ and hence a continuous function $f$ on $\mathbb C^n$ can be expressed as
\begin{align}\label{exp21}
f(\rho\omega)=\sum_{p,q}\sum_{j=1}^{d_{p,q}}~a_j^{p,q}(\rho)~Y_j^{p,q}(\omega),
\end{align}
where $\rho>0,~\omega\in S^{2n-1},$ and $~a_j^{p,q} $ are called the spherical
harmonic coefficients of $f$. The $(p,q)^{th}$ projection of $f$ is
given by
\begin{align}\label{exp22}
\Pi_{p,q}(f)(\rho,\omega)=~\sum_{j=1}^{d_{p,q}}~a_j^{p,q}(\rho)~Y_j^{p,q}(\omega).
\end{align}

We need the following lemma to decompose a homogeneous
polynomial into homogeneous harmonic polynomials.

\begin{lemma}\emph{\cite{T2}}\label{lemma2}
Every $P\in P_{p,q}$ can be uniquely expressed as
$P(z)=P_0(z)+|z|^2P_1(z)+\cdots+|z|^{2l}P_l(z),$ where $P_k\in
H_{p-k, q-k}$ and $l\leq min(p,q).$
\end{lemma}

\begin{corollary}\label{corol1}\emph{\cite{RS}}
Let $P\in H_{p,q}.$ Then it follows that
\begin{align*}
\bar{z_j}P(z)=P_0(z)+\gamma_{p,q}|z|^2\frac{\partial{P}}{\partial z_j}, \,
z_jP(z)=P_0'(z)+\gamma_{p,q}|z|^2\frac{\partial{P}}{\partial \bar{z}_j},
\end{align*}
where $\gamma_{p,q}=\frac{1}{(n+p+q-1)}$, $P_0\in H_{p, q+1}$ and $P_0'\in H_{p+1, q}.$
\end{corollary}

\section{Results on M\'{e}tivier groups}\label{section3}
\subsection{Characterization of certain continuous functions}\label{subsec}
As we know that the $\lambda$-twisted spherical mean of the M\'{e}tivier group
is not $U(n)$-invariant, we require to modify the space $Z_{r, R}$ appropriately.

\smallskip
We first recall the following fact from Geller \cite{GE}. The operator analog of a
bi-graded harmonic polynomial can be identical to the polynomial itself. Now for
our purpose, we assume $P_j^{p,q}(z)=|z|^{p+q}Y_j^{p,q}(\frac{z}{|z|})$ contains the
term $z^\alpha \bar{z}^\beta,$ for some multi-index $\alpha,\beta\in\mathbb Z_+^n$
with $|\alpha|=p$ and $|\beta|=q.$ By abuse of notation, we denote
\begin{align}\label{exp40}
P_j^{p,q}(Z)=Z^\alpha\bar{Z}^\beta,
\end{align} where $Z^\alpha=(Z_1^\lambda)^{\alpha_1}\cdots (Z_n^\lambda)^{\alpha_n}$
and $\bar{Z}^\beta =(\bar{Z}_1^\lambda)^{\beta_1}\cdots (\bar{Z}_n^\lambda)^{\beta_n}.$
Let $\tilde{a}_j^{p,q}(\rho)=\rho^{-(p+q)}a_j^{p,q},$ where $a_j^{p,q}$ as appears in (\ref{exp21}).
Let $Z^\ast_{r, R}$ be the space of smooth functions $f$ on $B_{r, R}$ satisfying  the conditions
\[\Pi_{0,0}\left(P_j^{0,q}(Z)\left(\Pi_{0,q}\left(P_j^{p,0}(Z)\left(\tilde{a}_j^{p,q}P_j^{p,q}
\right)\right)\right)\right)\in Z_{r, R}\] for all $p,q\in \mathbb{Z}_+$ and $1\leq j\leq d_{p,q}.$

\smallskip

Now, we fix some notations for our convenience. Denote $D_j=\rho\frac{\partial}{\partial\rho}+\frac{\nu_j}{2}\rho^2$
and $\bar{D}_j=\rho\frac{\partial}{\partial\rho}-\frac{\bar{\nu}_j}{2}\rho^2,$
where $\nu_j$ is defined in (\ref{exp6}). For multi-index $\alpha,\beta,$ define
{\footnotesize
\[D^\alpha=\prod_{i_1=1}^{\alpha_1}(\kappa_{1,i_1}D_1+2)\cdots\prod_{i_n=1}^{\alpha_n}(\kappa_{n,i_n}D_n+2)
\text{ and }
\bar{D}^\beta=\prod_{j_1=1}^{\beta_1}(\tilde{\kappa}_{1,j_1}\bar{D}_1+2)\cdots\prod_{j_n=1}^{\beta_n}
(\tilde{\kappa}_{n,j_n}\bar{D}_n+2),\]}
where $\kappa_{l,i_l},\tilde{\kappa}_{k,j_k}\in \{\gamma_{p',q'}=\frac{1}{(n+p'+q'-1)}:0\leq p'\leq p,0\leq q'\leq q\}.$

\smallskip

In order to prove the result for the functions in $Z^*_{r, R},$ it would be enough to
consider the following theorem.
\begin{theorem}\label{th2}
Let $f(z)=\tilde{a}(\rho)P_{p,q}(z),$ where $\rho=|z|$ and $P_{p,q}\in H_{p,q}.$ Then a
necessary condition for $f\in Z^\ast_{r, R}$ is that $\tilde{a}$ satisfies the ODE
\[\left(\sum\limits_{|\beta|+k=q}d_{\beta,k}\rho^{2k}\bar{D}^\beta\right)
\left(\sum\limits_{|\alpha|+l=p}c_{\alpha,l}\rho^{2l}D^\alpha \right) \tilde{a}=0\]
for some scalars $c_{\alpha,l},d_{\beta,k}\in \mathbb C.$

In particular, if $P_{p,q}(z)=z_{l_1}^p\bar{z}_{l_2}^q$ for some $1\leq l_1,l_2 \leq n,$
then there exist $A_i,B_k\in\mathbb{C}$ such that
\[\tilde{a}(\rho)=\sum_{i=0}^p A_ie^{-\frac{\nu_{l_1}}{4}\rho^2}\rho^{-2(p+q+n-i)}
+\sum_{k=0}^q B_ke^{\frac{\nu_{l_2}}{4}\rho^2}\rho^{-2(p+q+n-k)},\]
where $r<\rho<R$ and $A_0=B_0=0.$
\end{theorem}

\begin{proof}
For $p=q=0,$ we have $\tilde{a}(\rho)=\tilde{a}\times_\lambda\mu_\rho(0)=0,$
whenever $r<\rho<R.$
To proceed the other cases, we need to apply the operator $Z_j^{\lambda}$ to $f$.
\[Z_j^{\lambda}f=\dfrac{\partial f}{\partial z_j}+\frac{1}{4}\nu_j\bar{z}_j f+\frac{1}{4}
\sum_{\substack{l=1 \\ (l\neq j)}}^n (\eta_lz_l+\nu_l\bar{z}_l)f.\]
Given that $f=\tilde{a}P,$ the above equation will take the form
\begin{equation}\label{exp25}
Z_j^{\lambda}f=\frac{1}{2\rho^2}(D_j\tilde{a})\bar{z}_jP(z)+\tilde{a}\dfrac{\partial P}{\partial z_j}
+\frac{1}{4}\sum_{l\neq j}(\eta_lz_l+\nu_l\bar{z}_l)\tilde{a}P(z).
\end{equation}
Substituting the values of $\bar{z}_jP(z)$ and $z_jP(z)$ from Corollary \ref{corol1}, we have
\begin{align*}
Z_j^{\lambda}f &=\frac{1}{2\rho^2}D_j\tilde a\left(P_0+\gamma_{p,q}|z|^2\frac{\partial{P}}{\partial z_j}\right)
+\tilde{a}\frac{\partial P}{\partial z_j} \\
&+\frac{1}{4}\sum_{l\neq j}\left[\eta_l \tilde{a}\left(P_0'+\gamma_{p,q}|z|^2
\frac{\partial{P}}{\partial \bar{z_l}}\right)+\nu_l\tilde{a}\left(P_0+\gamma_{p,q}|z|^2
\frac{\partial{P}}{\partial z_l}\right)\right].
\end{align*}
After rearranging the terms, we get
\begin{align*}
Z_j^{\lambda}f&=\frac{1}{2\rho^2}D_j\tilde{a}P_0
+\frac{1}{4}\sum_{l\neq j}\nu_l\tilde{a}P_0 +\frac{1}{4}\sum_{l\neq j}\eta_l \tilde{a}P_0' \\
&+\frac{1}{2}\left(\gamma_{p,q}D_j+2\right)\tilde{a}\dfrac{\partial P}{\partial z_j}
+\frac{1}{4}\rho^2\gamma_{p,q}\sum_{l\neq j}\left(\eta_l\frac{\partial P}{\partial \bar{z}_l}
+\nu_l\frac{\partial P}{\partial {z}_l}\right)\tilde{a}.
\end{align*}
Now, the projection $\Pi_{p-1,q}$ of $Z^\lambda_jf$ is given by
\begin{align*}
\Pi_{p-1,q}~Z_j^{\lambda}f=\frac{1}{2}\left(\gamma_{p,q}\,D_j+2\right)\tilde{a}\dfrac{\partial P}{\partial z_j}
+\frac{1}{4}\rho^2\gamma_{p,q}\sum_{l\neq j} \nu_l\dfrac{\partial P}{\partial {z}_l}\tilde{a}.
\end{align*}
If $p=1$ and $q=0,$ then $\dfrac{\partial P}{\partial z_j}$ is a non-zero constant for some $j$, say $\zeta_j.$ Thus,
\begin{align}\label{exp26}
\Pi_{0,0}~Z_j^{\lambda}f&=\left\{\dfrac{1}{2n}\left(\rho\dfrac{\partial}{\partial\rho}
+\dfrac{\nu_j}{2}\rho^2 \right)+1\right\}\tilde{a}~\zeta_j
+\dfrac{1}{4}\sum_{l\neq j}\dfrac{\rho^2 \nu_l}{n}\tilde{a}~\zeta_l \\
&=\zeta_j\left\{\dfrac{1}{2n}\left(\rho\dfrac{\partial}{\partial\rho}
+\left(\dfrac{\nu_j}{2}+\sum_{l\neq j}\frac{ \nu_l}{2}\frac{\zeta_l}{\zeta_j}\right)\rho^2 \right)+1\right\}\tilde{a} \nonumber\\
&=\zeta_j\left\{\dfrac{1}{2n}\left(\rho\dfrac{\partial}{\partial\rho}+\frac{d_1}{2}\rho^2 \right)
+1\right\}\tilde{a}(\rho), \nonumber
\end{align}
where $d_1=(\nu_j+\sum_{l\neq j}\nu_l\frac{\zeta_l}{\zeta_j}).$ By definition of $Z^*_{r, R},$
$\Pi_{0,0}(Z_j^{\lambda}f) \in Z_{r, R}.$ Evaluating $\lambda$-twisted spherical
mean of $\Pi_{0,0}(Z_j^{\lambda}f)$ at $z=0$, we get
\[\left\{\frac{1}{2n}\left(\rho\dfrac{\partial}{\partial\rho}
+\frac{d_1}{2}\rho^2 \right)+1\right\}\tilde{a}(\rho)=0.\]
By replacing $\tilde{a}(\rho)=e^{-\frac{d_1}{4}\rho^2}\tilde{a}'(\rho)$ in the above equation, we get
\[e^{-\frac{d_1}{4}\rho^2}\left\{\frac{1}{2n}\rho\dfrac{\partial}{\partial\rho}+1\right\}\tilde{a}'(\rho)=0.\]
Thus, for $p=1,q=0$ we infer that
\[\tilde{a}(\rho)= A_1e^{-\frac{d_1}{4}\rho^2}\rho^{-2n}.\]
However, for the case $q=0$ and $p\geq 2,$ it would be difficult to solve the ODE. For instance,
consider $P(z)=z_1z_2.$ Then, after applying $Z_1Z_2$ to $\tilde a P,$ we get
\[\left\{(\gamma_{1,0}D_1+2)(\gamma_{2,0}D_2+2)+c_1c_2\rho^4\right\}\tilde{a}=0,\]
which we yet to solve.

For $p\geq 2$ and $q=0,$ the function  $\Pi_{0,0}\left(P^{p,0}(Z)\left(\tilde{a}^{p,0}P^{p,0}\right)\right)\in Z_{r, R},$
by evaluating its $\lambda$-twisted spherical mean at $z=0,$ we can infer that $\tilde{a}$ satisfies
\[\sum\limits_{|\alpha|+l=p}c_{\alpha,l}\rho^{2l}D^\alpha\tilde{a}=0.\]

By similar argument for $p=0$ and $q\geq 1,$ we get
\[\sum\limits_{|\beta|+k=q}d_{\beta,k}\rho^{2k}\bar{D}^\beta\tilde{a}=0.\]
In general, while $p,q\geq 1,$ we conclude that
\[\left(\sum\limits_{|\beta|+k=q}d_{\beta,k}\rho^{2k}\bar{D}^\beta\right)
\left(\sum\limits_{|\alpha|+l=p}c_{\alpha,l}\rho^{2l}D^\alpha \right) \tilde{a}=0.\]

However, if $P(z)$ is of the form $z_{l_1}^p\bar{z}_{l_2}^q$, then we can express $\tilde{a}$
explicitly as earlier. For this, first consider the case $q=0$ and $p\geq 1.$ Since
$\tilde a\,z_{l_1}^p\in Z^*_{r, R},$ it follows that $\Pi_{0,0}Z_{l_1}^p(\tilde{a}P)\in Z_{r, R}.$
Thus, by evaluating $\lambda$-twisted spherical mean of $\Pi_{0,0}Z_{l_1}^p(\tilde{a}P)$ at $z=0,$ we get
\[\prod_{i=1}^p\left(\gamma_{p-(i-1),0}~D_{l_1}+2\right)\tilde{a}=0.\]
This, in turn, implies that
\[\tilde{a}(\rho)=\sum^p_{i=1}A_ie^{-\frac{\nu_{l_1}}{4}\rho^2}\rho^{-2(n+p-i)}.\]
Similarly, for $p=0$ and $q\geq 1,$ by considering the operator  $\bar{Z}_{l_2}^\lambda,$
we can derive that
\[\tilde{a}(\rho)=\sum_{k=1}^q B_ke^{\frac{\bar{\nu}_{l_2}}{4}\rho^2}\rho^{-2(n+q-k)}.\]
If $p, q\geq1,$ by evaluating $\lambda$-twisted spherical mean of
$\Pi_{0,0}\bar{Z}_{l_2}\Pi_{0,q}Z_{l_1}^p(\tilde{a}P)$ at $z=0,$ we obtain
\[\prod_{k=1}^q\left(\gamma_{p,q+(k-1)}~\bar{D}_{l_2}+2\right)
\prod_{i=1}^p\left(\gamma_{p-(i-1),q}~D_{l_1}+2\right)\tilde{a}=0.\]
Hence, a solution to the above equation can be expressed as
\[\tilde{a}(\rho)=\sum_{i=1}^p~A_{i}~e^{-\frac{\nu_{l_1}}{4}\rho^2}\rho^{-2(n+p+q-i)}+
~\sum_{k=1}^q~B_{k}~e^{\frac{\nu_{l_2}}{4}\rho^2}\rho^{-2(n+p+q-k))}.\]
This completes the proof.
\end{proof}

\begin{remark}\label{rk1}
In the definition of $Z^*_{r, R}$ we have assumed that, for all $p,q\in \mathbb{Z}_+$ and $1\leq j\leq d_{p,q},$
\begin{align}\label{exp17}
\Pi_{0,0}\left(P_j^{0,q}(Z)\left(\Pi_{0,q}\left(P_j^{p,0}(Z)\left(\tilde{a}_j^{p,q}P_j^{p,q}
\right)\right)\right)\right)\times_\lambda\mu_s(z)=0
\end{align}
for all $z\in \mathbb C^n$ and $s >0$ with $S_s(z)\subseteq B_{r, R}$
and $B_r(0)\subseteq B_s(z).$ However, for a proof of Theorem \ref{th2},
it is enough to assume that (\ref{exp17}) holds for $z=0,$ whenever $r<s<R.$
Consequently, sufficient part of Theorem \ref{th2}, at $z=0,$ is obviously true.

Further, as compared to the Heisenberg group, it would be a reasonable question
to consider $e^{\frac{c}{4}|z|^2}|z|^{-2(n+p+q-i)}P(z)$ to be in $Z_{r,\infty}$
for appropriate choice of $c$ and $i,$ where $P\in H_{p,q}.$
In general, the matrix  $\sum_{j=1}^m \lambda_jU^{(j)},$ arises from the
symplectic form, has distinct eigenvalues, makes the difficulty to find out
such $c.$ However, in the case of $H$-type group, all the eigenvalues
are identical, we have such a result in Section \ref{section4}, Theorem \ref{th6}.

\end{remark}

\subsection{Injectivity and support theorem}
In this section, we simplify the $\lambda$-twisted spherical mean on
the M\'{e}tivier group to another mean, which is similar to the TSM
on the Heisenberg group. This will ease to prove support theorem for
the $\lambda$-twisted spherical mean on  M\'{e}tivier group, for the
type function. Further, we prove that a non-harmonic complex cone
aligned with one of the coordinate axes in $\mathbb{C}^n$ is a set
of injectivity for the $\lambda$-twisted spherical mean on the
M\'{e}tivier groups.

\smallskip

For $\lambda\in \mathbb{R}^m\setminus \{0\},$ the skew symmetric matrix
$V_\lambda=\sum_{j=1}^m \lambda_jU^{(j)}$ is non-singular (see \cite{MS}).
Let $u_1\pm iv_1,\ldots,u_n\pm iv_n$ be the eigenvectors of $V_\lambda$
with corresponding eigenvalues $\pm i\mu_{\lambda,1},\ldots,\pm i\mu_{\lambda,n},$
where $\mu_{\lambda,1}\geq \cdots \geq\mu_{\lambda,n} > 0.$ Define
$A_\lambda=(\sqrt 2~ v_1,\ldots,\sqrt 2~v_n,\sqrt 2~u_1,\ldots,\sqrt 2~u_n).$ Then
$A_\lambda$ is an orthogonal matrix and satisfies $V_\lambda A_\lambda=A_\lambda U_\lambda,$
where
\begin{align}\label{exp11}
U_\lambda=\left(
\begin{array}{cc}
0_n & -J_\lambda \\
J_\lambda  & 0_n
\end{array}
\right)
\end{align}
with $J_\lambda=\text{diag}(\mu_{\lambda,1},\ldots,\mu_{\lambda,n})$ and $0_n$ is zero matrix of order $n.$
Thus, in view of (\ref{exp11}), we have
\[\sum_{j=1}^m\lambda_j\langle x,U^{(j)}\xi\rangle=\langle x,V_\lambda\xi\rangle=\langle A_\lambda^tx, U_\lambda A_\lambda^t\xi\rangle,\]
where $A_\lambda A_\lambda^t=I.$ That is,
\begin{equation}\label{exp12}
\sum_{j=1}^m\lambda_j\text{Re}\,(z\cdot\overline{U^{(j)}w})
=\sum_{j=1}^n\mu_{\lambda,j} \text{ Im}\left((z_\lambda)_j\cdot(\bar{w}_\lambda)_j\right),
\end{equation}
where $z_\lambda$ and $w_\lambda$ are complexification of $A_\lambda^tx$ and $A_\lambda^t\xi$ respectively.

\smallskip

Let $f\in L^1(\mathbb{C}^n),$ then define
\begin{align}\label{exp41}
f_\lambda(z)=f(\tilde{z}_\lambda),
\end{align}
where $z,\tilde{z}_\lambda\in \mathbb C^n$ be the complexification of $x,A_\lambda x \in \mathbb R^{2n}$
respectively. The following lemma would give a simplification of the $\lambda$-twisted spherical mean on
the M\'{e}tivier group defined by (\ref{exp14}).

\begin{lemma}\label{lemma4}
Let $f\in L^1(\mathbb{C}^n)$ and $f_\lambda$ be as in (\ref{exp41}). Then
$f\times_\lambda\mu_s(\tilde{z}_\lambda)=f_\lambda \tilde{\times}_\lambda\mu_s(z),$ where
\begin{align}\label{exp15}
f_\lambda \tilde{ \times}_\lambda \mu_s(z)=\int_{|w|=s}~f_{\lambda}(z-w)
~e^{\frac{i}{2}\sum_{j=1}^n\mu_{\lambda,j} \text{ Im}\left(z_j\cdot\,\bar{w}_j\right)}~d\mu_s(w).
\end{align}
\end{lemma}

\begin{proof}
In view of (\ref{exp16}), we can write
\begin{align*}
f \times_\lambda \mu_s(A_\lambda x)
&=\int_{|\xi|=s}~f(A_\lambda x-\xi)~e^{\frac{i}{2}\langle A_\lambda x,A_\lambda U_\lambda A_\lambda^t\xi\rangle}~d\mu_s(\xi)\\
&=\int_{|\xi|=s}~f_\lambda(x-A_\lambda^t\xi)~e^{\frac{i}{2}\langle x, U_\lambda A_\lambda^t\xi\rangle}~d\mu_s(\xi)\\
&=\int_{|\xi|=s}~f_\lambda(x-\xi)~e^{\frac{i}{2}\langle x, U_\lambda\xi\rangle}~d\mu_s(\xi)\\
&=\int_{|w|=s}~f_{\lambda}(z-w)
~e^{\frac{i}{2}\sum_{j=1}^n\mu_{\lambda,j} \text{ Im}\left(z_j\cdot\,\bar{w}_j\right)}~d\mu_s(w)\\
&=f_\lambda \tilde{ \times}_\lambda \mu_s(z).
\end{align*}
\end{proof}

To deal with the modified $\lambda$-twisted spherical mean $f_\lambda \tilde{ \times}_\lambda \mu_s,$
defined in (\ref{exp15}), it is required to study the function $f_\lambda.$ In particular, we need to
find out those polynomials $P$ such that $P_\lambda\in H_{p,q}$ for some $p,q.$ Let $P_\lambda\in H_{p,q}.$
By identifying $\mathbb C^n$ with $\mathbb R^{2n},$ we get $P_\lambda\in H_l,$ where $l=p+q,$ and $H_l$
is the space of all homogeneous harmonic polynomials of degree $l$ on $\mathbb{R}^{2n}.$
Since $P(x)=P_\lambda(A_\lambda^tx)$ and the Laplacian is rotation invariant, we get $P\in H_l,$
and hence \[P\in \bigoplus\limits_{p'+q'=l} H_{p',q'}.\] With the above observation, we define
\begin{align}\label{exp42}
H^\lambda_{p,q}=\{P\in \bigoplus_{p'+q'=p+q} H_{p',q'}: P_\lambda\in H_{p,q}\}.
\end{align}
Next, we prove a similar result to support theorem for the M\'{e}tivier groups. Consider the following
left-invariant differential operators for the $\lambda$-twisted spherical mean (\ref{exp15}),
\[\tilde{Z}_j^\lambda=\frac{\partial}{\partial z_j}-\frac{\mu_{\lambda,j}}{4}{\bar{z}}_j~\mbox{and }
\tilde{Z}^{*^\lambda}_j=\frac{\partial}{\partial \bar{z}_j}+\frac{\mu_{\lambda,j}}{4}{z}_j,~ j=1,2,...,n.\]
Since $P_\lambda\in H_{p,q},$ define $P_\lambda^{p,q}(\tilde{Z})$ as in (\ref{exp40}), replacing $Z$
by $\tilde{Z}.$

\begin{theorem}
Let $f=\tilde{a}P,$ where $P\in H^\lambda_{p,q},$ be a smooth function on $\mathbb C^n$ and
$|z|^k e^{\frac{\mu_{\lambda,1}}{4}|z|^2}f(z)$ is bounded for each $k\in\mathbb{Z}_+.$
Then $\Pi_{0,0}\left(P_\lambda^{p,q}(\tilde{Z})f_\lambda\right)\tilde{\times}_\lambda\,\mu_s(z)=0$
for all  $z\in\mathbb C^n$ and $s>r+|z|$ if and only if $f$ is supported in $|z| \leq r.$
\end{theorem}
\begin{proof}
If $f=\tilde{a}P,$ then $f_\lambda=\tilde{a}P_\lambda.$ Clearly for $p=q=0$, $\tilde{a}=0.$
Let $p\geq 1,$ then applying $\tilde{Z}_j^\lambda$ to $f_\lambda,$ we have
\begin{align*}
\tilde{Z}_j^\lambda f_\lambda=\frac{1}{2}\left(\frac{1}{\rho}\dfrac{\partial}{\partial\rho}
-\frac{\mu_{\lambda,j}}{2}\right)\tilde{a}\bar{z}_jP_\lambda+\tilde{a}\dfrac{\partial P_\lambda}{\partial z_j}.
\end{align*}
Since $P_\lambda\in H_{p,q},$ substituting the value of $\bar{z}_jP_\lambda$ from corollary \ref{corol1},
we have
\begin{align*}
\tilde{Z}_j^\lambda f_\lambda=\frac{1}{2}\left(\frac{1}{\rho}\dfrac{\partial}{\partial\rho}
-\frac{\mu_{\lambda,j}}{2}\right)\tilde{a}P_{\lambda,0}
+\left[\frac{\gamma_{p,q}}{2}\left(\rho\dfrac{\partial}{\partial\rho}
-\rho^2\frac{\mu_{\lambda,j}}{2}\right)+1\right]\tilde{a}\dfrac{\partial P_\lambda}{\partial z_j},
\end{align*}
where $\gamma_{p,q}=\frac{1}{(n+p+q-1)}.$

Consider $q=0$ and $p=1.$ Then there exists a $j_o$ such that
$\dfrac{\partial P_\lambda}{\partial z_{j_o}}\neq 0.$ Thus, from
the given condition that
$\Pi_{0,0}(\tilde{Z}_{j_o}^\lambda f_\lambda)\tilde{\times}_\lambda\,\mu_s(0)=0$ for all $s>r,$
we arrived at
\begin{align*}
\left[\frac{1}{2n}\left(\rho\dfrac{\partial}{\partial\rho}
-\rho^2\frac{\mu_{\lambda,j_o}}{2}\right)+1\right]\tilde{a}=0,
\end{align*}
whenever $\rho>r.$ This leads to a solution
\begin{align*}
\tilde{a}(\rho)= A_1e^{-\frac{\mu_{\lambda,j_o}}{4}\rho^2}\rho^{-2n}.
\end{align*}
By an induction argument, for $q=0$ and $p\geq 1,$ from
$\Pi_{0,0}\left(P_\lambda^{p,0}(\tilde{Z})f_\lambda\right)\tilde{\times}_\lambda\,\mu_s(0)=0$
for all $s>r,$ it follows that
\begin{align*}
\prod_{i=1}^p\left\{\frac{1}{2(n+p-i)}\left(\rho\dfrac{\partial}{\partial\rho}
-\rho^2\frac{\mu_{\lambda,j_i}}{2}\right)+1\right\}\tilde{a}=0.
\end{align*}
Solving the above equation we get
\[\tilde{a}_{p,0}(\rho)=\sum_{i=1}^p A_ie^{-\frac{c_i}{4}\rho^2}\rho^{-2(p+n-i)},\]
where $c_i\in\{\mu_{\lambda,j}:1\leq j \leq n\}.$ Similar conclusion holds true for $p=0,q\geq 1.$

In general, for $p,q\geq 1,$ $\tilde{a}$ satisfies the ODE
\begin{align*}
\prod_{k=1}^q\left\{\frac{\gamma_{p,q+1-k}}{2}\left(\rho\dfrac{\partial}{\partial\rho}
+\rho^2\frac{d_k}{2}\right)+1\right\}
\prod_{i=1}^p\left\{\frac{\gamma_{p+1-i,q}}{2}\left(\rho\dfrac{\partial}{\partial\rho}
-\rho^2\frac{c_i}{2}\right)+1\right\}\tilde{a}=0,
\end{align*}
and be expressed as
\begin{equation}\label{exp24}
\tilde{a}_{p,q}(\rho)=\sum_{i=1}^p A_ie^{-\frac{c_i}{4}\rho^2}\rho^{-2(p+q+n-i)}
+\sum_{k=1}^q B_ke^{\frac{d_k}{4}\rho^2}\rho^{-2(p+q+n-k)}
\end{equation}
for all $\rho>r,$ where $c_i,d_k\in\{\mu_{\lambda,j}:1\leq j \leq n\}$ and $A_i,B_k$ are constants.
Since $\mu_{\lambda,1}\geq \mu_{\lambda,j}>0$ for all $j,$ by the given growth conditions, we infer
that $f_\lambda(z)=0$ for all $|z|>r.$ Thus, we conclude that $f$ is supported in $|z|\leq r.$
\end{proof}

A set $K\subset\mathbb C^n~(n\geq 2),$ which is closed under complex scaling,
is known as a complex cone. Further, a complex cone that does not intersect
the zero set of any bi-graded homogeneous harmonic polynomial is called
{\em non-harmonic}. The zero set of the polynomial $H(z)=az_1\bar z_2+|z|^2,$
where $a\not=0$ and $z\in\mathbb C^n$ is a non-harmonic complex cone, (see \cite{Sri4}).

\smallskip

 Let $z\in \mathbb C^n$ be the complexification of $x\in \mathbb R^{2n},$
and $\tilde{z}_\lambda$ be the complexification of $A_\lambda x.$ For a complex cone $K,$ define
$K_\lambda=\{z\in\mathbb C^n:\tilde{z}_\lambda\in K\}.$ Then $K_\lambda$ is also a complex cone,
and $K$ is non-harmonic if and only if $K_\lambda$ is non-harmonic.

\begin{theorem}
Suppose $K$ is a non-harmonic complex cone such that $K_\lambda$ aligned with one of
the coordinate axes in $\mathbb{C}^n~(n\geq 2).$ Let $f$ be a continuous function
on $\mathbb{C}^n$ such that $f\times_\lambda\,\mu_r(z)=0,$ for all $r>0$ and $z\in K.$
Then $f=0.$
\end{theorem}

\begin{proof}
In view of Lemma \ref{lemma4},  given $f\times_\lambda\mu_r=0$ on $K$ implies
$f_\lambda\tilde{\times}_\lambda\,\mu_r=0$ on $K_\lambda.$
By hypothesis, without loss of generality, we can assume  $z=(z_1,0,...,0)\in K_\lambda$
for all $z_1\in\mathbb{C}.$ Thus,
\[\int_{|w|\leq r}f_\lambda(z+w)e^{-\frac{i}{2}\sum_{j=1}^n\mu_{\lambda,j} \text{ Im}\left(z_j\cdot\,\bar{w}_j\right)}dw
=\int_0^rf_\lambda\tilde{\times}_\lambda~\mu_s(z)s^{2n-1}ds= 0\]
for all $r>0$ and $z\in K_\lambda.$ Applying $2\partial_{z_1}$ to the above equation, we get
\begin{align*}
\int_{|w|\leq r}\frac{\partial}{\partial w_1}\Big(f_\lambda(z+w)&e^{-\frac{i}{2}\sum_{j=1}^n\mu_{\lambda,j}
\text{ Im}\left(z_j\cdot\,\bar{w}_j\right)}\Big)dw\\
&-\frac{\mu_{\lambda,1}}{2}\int_{|w|\leq r}\bar{w}_1 f_\lambda(z+w)e^{-\frac{i}{2}\sum_{j=1}^n\mu_{\lambda,j}
\text{ Im}\left(z_j\cdot\,\bar{w}_j\right)}ds=0.
\end{align*}
It follows by an application of Green's theorem that
\begin{align*}
\int_{|w|=r}\frac{\bar{w}_1}{r}\Big(f_\lambda(z+w)&e^{-\frac{i}{2}\sum_{j=1}^n\mu_{\lambda,j}
\text{ Im}\left(z_j\cdot\,\bar{w}_j\right)}\Big)dw\\
&=\frac{\mu_{\lambda,1}}{2}\int_{|w|\leq r}\bar{w}_1 f_\lambda(z+w)e^{-\frac{i}{2}\sum_{j=1}^n\mu_{\lambda,j}
\text{ Im}\left(z_j\cdot\,\bar{w}_j\right)}dw.
\end{align*}
Let $F(t)=t^{2n-1}g~\tilde{\times}_\lambda~\mu_t(z),$ where $g(z)=\bar{z}_1f_\lambda(z).$
Then we have
\begin{equation}\label{exp29}
\frac{F(r)}{r}=\frac{\mu_{\lambda,1} }{2}\int_0^rF(s)ds.
\end{equation}

It is easy to see that (\ref{exp29}) satisfies the ODE
\[F^{'}(r)=\left(\frac{\mu_{\lambda,1} r}{2}+\frac{1}{r}\right)F(r)\]
having the general solution
\[F(r)=\frac{c(z)}{r}e^\frac{\mu_{\lambda,1} r^2}{4}.\]
That is,
\[r^{2n-2}g~\tilde{\times}_\lambda~\mu_r(z)=c(z)e^\frac{\mu_{\lambda,1} r^2}{4}.\]
Letting $r\rightarrow 0,$ we get $c(z)=0$. Hence $\bar{z}_1f_\lambda~\tilde{\times}_\lambda~\mu_r(z)=0$
for all $r>0$ and $z\in K_\lambda.$ By replicating the above procedure, we get $(Pf_\lambda)\tilde{\times}_\lambda~\mu_r(z)=0$
for arbitrary polynomial $P(z_1,\bar z_1).$ By a similar argument as in
(\cite{Sri4}, Lemma 2.6) we conclude that $f_\lambda=0$ and hence $f=0.$
\end{proof}

\begin{remark}
If we consider $H$-type group instead of the general M\'{e}tivier group, then the
restriction on the cone to align with one of the coordinates axes could be relaxed.
\end{remark}

\section{Some results on $H$-type groups}\label{section4}
In this section, we see that the $\lambda$-twisted spherical mean on the $H$-type group
can be related to the $|\lambda|$-twisted spherical mean on the Heisenberg group. Although
the $\lambda$-twisted spherical mean on $H$-type group is not $U(n)$-invariant,
we can prove sufficient condition for a function to be in $Z_{r,\infty},$ and
an analogue of Helgason's support theorem together with Heche-Bochner identity for the $H$-type
group. Further, we prove that the boundary of a bounded domain is a set of injectivity
for $\lambda$-twisted spherical mean on the $H$-type group.

\smallskip

We know that for the $H$-type groups, $\mu_{\lambda,j}=|\lambda|$ for all $j,$ due to the fact that
$\sum_{j=1}^m\lambda_jU^{(j)}=|\lambda|V.$
Thus (\ref{exp12}) becomes
\begin{align*}
\sum_{j=1}^m\lambda_j\text{Re}\,(z\cdot\overline{U^{(j)}w})
=|\lambda|\text{ Im}\left(z_\lambda\cdot\bar{w}_\lambda\right)
\end{align*}
and from Lemma \ref{lemma4}, we have
\begin{align}\label{exp13}
f\times_\lambda\mu_s(z)=f_\lambda\tilde{\times}_\lambda~\mu_s(z_\lambda)=f_\lambda \times_{|\lambda|} \mu_s(z_\lambda),
\end{align}
where $f_\lambda \times_{|\lambda|} \mu_s$ is the $|\lambda|$-twisted spherical mean on
the Heisenberg group. Similarly, the $\lambda$-twisted convolution on the $H$-type group
can be related to the twisted convolution on the Heisenberg group by
\begin{align}\label{exp27}
f\times_\lambda g(z)=f_\lambda \times_{|\lambda|} g_\lambda(z_\lambda).
\end{align}

Next, we present the sufficient condition for functions to be in $Z_{r,\infty},$ which
we mentioned in Remark \ref{rk1}.
\begin{theorem}\label{th6}
Let $P\in H^\lambda_{p,q}$ and $h(z)=\dfrac{{e^{\frac{|\lambda|}{4}|z|^2}}P(z)}{|z|^{2(n+p+q-i)}},$
where $1\leq i\leq p$ and $H^\lambda_{p,q}$ is defined in (\ref{exp42}). Then $h\in Z_{r,\infty}.$
\end{theorem}
\begin{proof}
To prove the result, it needs to verify that $h\times_\lambda\mu_s (z)=0$ for all $z\in\mathbb C^n$
and $s>|z|+r.$ From (\ref{exp13}), it is enough to show that $h_\lambda\times_{|\lambda|} \mu_s(z)=0$
for all $z\in\mathbb C^n$ and $s>|z|+r.$

Let $\eta=n+p+q$ and consider
\begin{align*}
h_\lambda \times_{|\lambda|} \mu_s(z)=\int_{|w|=s}
\frac{e^{\frac{|\lambda|}{4}|z+w|^2}P_\lambda(z+w)}{|z+w|^{2(\eta-i)}}
e^{-\frac{i}{2}|\lambda|\text{ Im}\left(z\cdot\bar{w}\right)}d\mu_s(w).
\end{align*}
Simplifying the exponential terms, it is enough to show the following integral is zero
\[\int_{|w|=s}
\frac{e^{\frac{|\lambda|}{2}\bar{z}\cdot w}P_\lambda(z+w)}{|z+w|^{2(\eta-i)}}d\mu_s(w).\]
Again if we expand exponential term the above integral will reduce to
\[\int_{|w|=s}\frac{w^\alpha P_\lambda(z+w)}{|z+w|^{2(\eta-i)}}d\mu_s(w).\]
Hence we arrived at the same Euclidean situation, which was proved by the author in (\cite{RS}, Theorem 3.3).
Thus, it follows that $h_\lambda \times_{|\lambda|} \mu_s(z)=0$ for all $z\in\mathbb C^n$ and $s>|z|+r.$
\end{proof}

Next, we shall prove support theorem for the $\lambda$- twisted spherical mean
on the H-type group, for which we need to recall the following support theorem
for the TSM on the Heisenberg group.
\begin{theorem}\label{th11}\emph{\cite{RS}}
Let $g$ be a continuous function on $\mathbb C^n$ such that for each
$k\in\mathbb{Z}_+,$ $|z|^k e^{\frac{|\lambda|}{4}|z|^2}g(z)$ is bounded
for every $k\in\mathbb{Z}_+.$ Then $g$ is supported in $|z| \leq r$ if and only if $g\times_{|\lambda|}\mu_s(z)=0 $
for all $z\in\mathbb C^n$ and $s>r+|z|.$
\end{theorem}

Using (\ref{exp13}) we prove the following support theorem for the
$H$-type groups.

\begin{theorem}
Suppose $f$ is a continuous function on $\mathbb C^n$ such that  each of the function
$|z|^ke^{\frac{|\lambda|}{4}|z|^2}f(z)$  is bounded. Then $f$ is supported in $|z|\leq r$
if and only if $f \times_\lambda \mu_s(z) = 0 $ for all $z\in\mathbb C^n$ and $s>r+|z|.$
\end{theorem}
\begin{proof}
We know that $f_\lambda(z_\lambda)=f(z),$ where $z$ and $z_\lambda$ are the complexification of
$x$ and $A_\lambda^tx$ respectively. Since $|z|^k e^{\frac{|\lambda|}{4}|z|^2}f(z),$
is bounded, it follows that  $|z|^k e^{\frac{|\lambda|}{4}|z|^2}f_\lambda(z)$ is bounded because $|z|=|z_\lambda|.$
Now, $f$ is supported in $|z| \leq r$ if and only if $f_\lambda$ is supported in $|z| \leq r.$
Hence, from (\ref{exp13}) and Theorem \ref{th11}, we get the desired result.
\end{proof}

Now, we state the Heche-Bochner identity for the twisted convolution on the Heisenberg group.
Consider the Laguerre functions on $\mathbb{C}^n,$
\[\varphi_k^{n-1}(z)=L_k^{n-1}\left(\frac{1}{2}|z|^2\right)e^{-\frac{1}{4}|z|^2}.\]
For $\lambda\in\mathbb{R}^m\smallsetminus\{0\},$
define $\varphi_{k,\lambda}^{n-1}(z)=\varphi_k^{n-1}(|\lambda|^{1/2}z).$
\begin{theorem}\label{th7}\emph{\cite{T2}}
Let $f\in L^1(\mathbb{C}^n)$ be of the form $f=Pg,$ where $g$ is radial and $P\in H_{p,q}.$
Then for  $\lambda>0,$
\[f\times_\lambda\varphi_{k,\lambda}^{n-1}(z)=
\begin{cases}
(2\pi)^{-n}\lambda^{p+q}P(z)g\times_\lambda\varphi_{k-p,\lambda}^{n+p+q-1}(z), &\text{ if } k\geq p \\
0, &\text{ otherwise}
\end{cases}\]
and for $\lambda<0,$
\[f\times_\lambda\varphi_{k,\lambda}^{n-1}(z)=
\begin{cases}
(2\pi)^{-n}|\lambda|^{p+q}P(z)g\times_\lambda\varphi_{k-q,\lambda}^{n+p+q-1}(z), &\text{ if } k\geq q \\
0, &\text{ otherwise},
\end{cases}\]
where convolution on the right-hand side is on $\mathbb{C}^{n+p+q}.$
\end{theorem}

An analogue of the above Heche-Bochner identity for $H$-type group can be stated as follows.
\begin{theorem}
Let $f\in L^1(\mathbb{C}^n)$ be of the form $f=gP$ where $g$ is radial and $P\in H^\lambda_{p,q},$ where
$H^\lambda_{p,q}$ defined in (\ref{exp42}). Then for $\lambda\in\mathbb{R}^m\smallsetminus\{0\}$
\[f\times_\lambda\varphi_{k,\lambda}^{n-1}(z)=
(2\pi)^{-n}|\lambda|^{p+q}P(z)g\times_\lambda\varphi_{k-p,\lambda}^{n+p+q-1}(z'),\]
if $k\geq p$ and $0$ otherwise, where $z'\in \mathbb C^{n+p+q}$ be such that $|z|=|z'|$
and convolution on the right is on $\mathbb{C}^{n+p+q}.$
\end{theorem}

\begin{proof}
Since $\varphi_{k,\lambda}^{n-1}$ is radial, by (\ref{exp27}) and Theorem \ref{th7} we get
\begin{align*}
f\times_\lambda\varphi_{k,\lambda}^{n-1}(z)
&=f_\lambda \times_{|\lambda|} \varphi_{k,\lambda}^{n-1}(z_\lambda)\\
&=(2\pi)^{-n}|\lambda|^{p+q}P_\lambda(z_\lambda)g\times_{|\lambda|}\varphi_{k-p,\lambda}^{n+p+q-1}(z_\lambda') \\
&=(2\pi)^{-n}|\lambda|^{p+q}P(z)g\times_\lambda\varphi_{k-p,\lambda}^{n+p+q-1}(z'),
\end{align*}
where $z_\lambda',z'\in \mathbb C^{n+p+q}$ such that $|z_\lambda|=|z_\lambda'|=|z'|.$
\end{proof}

Next, we deduce an injectivity result for the $H$-type groups, which is
known for the Heisenberg group.
\begin{theorem}\label{th100}\emph{\cite{AR}}
Let $\partial\Omega$ be the boundary of a bounded domain $\Omega$ in $\mathbb{C}^{n}.$
Let f  be such that $f(z) e^{\left(\frac{1}{4}+\epsilon\right) |z|^{2}}
\in L^{p}\left(\mathbb{C}^{n}\right),$ for some $\epsilon>0$  and $1 \leq p \leq \infty.$
Suppose that $f \times\mu_{s}(z)=0$ for all $z \in \partial\Omega$ and $s>0.$  Then $f=0.$
\end{theorem}
Now, we state an analogue of the above result for the $H$-type groups.
\begin{theorem}
Let $\partial\Omega$ be the boundary of a bounded domain $\Omega$ in $\mathbb{C}^{n}.$
Let f  be such that $f(z) e^{\left(\frac{1}{4}+\epsilon\right)
|\lambda||z|^2}\in L^p\left(\mathbb{C}^{n}\right),$
 for some $\epsilon>0$  and $1 \leq p \leq \infty.$ Suppose that
 $f \times_\lambda \mu_{s}(z)=0$ for all $z \in \partial\Omega$ and $s>0.$ Then $f=0.$
\end{theorem}
\begin{proof}
From (\ref{exp13}) we have $f \times_\lambda \mu_s(z)=f_\lambda \times_{|\lambda|}\mu_s(z_\lambda).$
Define $\Omega^\prime=\{z_\lambda\in\mathbb C^n:z\in\Omega\}.$ Since the boundary of bounded domain $\Omega^\prime$ is
$\partial\Omega^\prime=\{z_\lambda:~ z\in\partial\Omega\},$ by Theorem \ref{th100},
we can conclude that $f_\lambda=0$ and hence $f=0.$
\end{proof}

\noindent\textbf{Concluding remark:}
We know that in the case of the M\'{e}tivier groups, the symplectic bilinear
form $\sum_{j=1}^m\lambda_j\text{Re}\,(z\cdot\overline{U^{(j)}w})$ cannot be made $U(n)$
-invariant, because all of $\mu_{\lambda,j}$ need not be identical. Hence we require more
assumptions on the functions to prove analogous results as to the Heisenberg group.
However, in the case of the $H$-type groups, all $\mu_{\lambda,j}$ are identical,
we do not require further assumption to prove the results for $H$-type groups.

\bigskip

\noindent{\bf Acknowledgements:} The first and second authors would like to gratefully
acknowledge the support provided by IIT Guwahati, Government of India.

\end{document}